\documentclass[preprints,article,accept,moreauthors,pdftex]{mdpi}

\firstpage{1} 
\pubvolume{xx}
\issuenum{y}
\articlenumber{z}
\pubyear{2020}
\copyrightyear{2020}
\history{July 31, 2020} 

\usepackage{amsmath,amssymb}
\usepackage{graphicx}
\usepackage{algcompatible}
\usepackage[ruled,vlined]{algorithm2e}
\usepackage{array}
\usepackage{tabularx}

\renewcommand{\Re}{\mathbb{R}}
\DeclareMathOperator{\minimize}{minimize}
\DeclareMathOperator{\argmin}{argmin}
\def\st{\mbox{s.t.}}

\newtheorem{prop}{Proposition}

\ifpdf
\DeclareGraphicsExtensions{.eps,.pdf,.png,.jpg}
\else
\DeclareGraphicsExtensions{.eps}
\fi

\Title{Spatially Adaptive Regularization in Image Segmentation}

\Author{Laura Antonelli $^{1,\ddagger}$\orcidA{0000-0002-4031-099X}, Valentina De Simone $^{2,\ddagger}$\orcidA{0000-0002-3357-5252}
and Daniela di Serafino $^{2,\dagger,\ddagger}$\orcidA{0000-0001-8215-0771}}

\AuthorNames{Laura Antonelli, Valentina De Simone and Daniela di Serafino}

\address{%
$^{1}$ \quad Institute for High Performance Computing and Networking (ICAR), CNR, Naples, Italy; laura.antonelli@cnr.it\\
$^{2}$ \quad University of Campania ``Luigi Vanvitelli'', Department of Mathematics and Physics, Caserta, Italy; daniela.diserafino@unicampania.it,
   valentina.desimone@unicampania.it}

\corres{Correspondence: daniela.diserafino@unicampania.it}

\firstnote{Current address: Department of Mathematics and Physics, University of Campania ``Luigi Vanvitelli'',
viale A.~Lincoln 5, I-81100 Caserta, Italy} 
\secondnote{All the authors contributed equally to this work.}

\abstract{We modify the total-variation-regularized image segmentation model proposed by Chan, Esedo\=glu and Nikolova
[SIAM Journal on Applied Mathematics 66, 2006] by introducing local regularization that takes into account spatial
image information. We propose some techniques for defining local regularization parameters, based on the
cartoon-texture decomposition of the given image, on the mean and median filters, and on a thresholding technique,
with the aim of preventing excessive regularization in piecewise-constant or smooth regions and preserving spatial
features in nonsmooth regions. We solve the modified model by using split Bregman iterations.
Numerical experiments show the effectiveness of our approach.}



\begin{document}

\section{Introduction\label{sec:intro}}

Image segmentation is a fundamental problem in image processing and computer vision. 
Its goal is to divide the given image into regions that represent different objects in the scene.
Variational models for image segmentation have been widely investigated, proving to be very effective
in many applications -- see, e.g.,~\cite{bib:WangEtAl2020} and the references therein.
Roughly speaking, the segmentation may be obtained by minimizing a cost functional which linearly
combines regularization and data fidelity terms:
\begin{equation}
\label{eq:var_model}
  \min_{u \in \Omega}  \, E(u) \, \equiv \, F(u)+\lambda G(u),
\end{equation}
where $u: \Omega \subset \Re^2 \rightarrow \Re$ provides a representation of the segmentation and $\lambda > 0$
is a parameter that controls the weight of the fidelity term $G$ versus the regularization term $F$. 
A widely-used segmentation model is the two-phase partitioning model introduced by Chan, Esedo\=glu and
Nikolova~\cite{bib:CEN2006}, which we refer to as CEN model:
\begin{equation} \label{eq:cen_continuous}
\begin{array}{ll}
\displaystyle \underset{u, c_1, c_2}{\minimize} & \displaystyle \int_{\Omega} | \nabla u | \, dx
                        + \lambda \int_{\Omega} \left( (c_1 - \bar u (x))^2 u(x) + (c_2 - \bar u (x))^2  (1-u(x)) \right) dx, \\
\st            & 0 \le u \le 1, \\[2pt]
                & c_1, c_2 > 0.
\end{array}
\end{equation}
Here $\bar u$ denotes the image to be segmented, which is assumed to take its values in $[0,1]$.
The CEN model allows us to obtain one of the two domains defining the segmentation,
$\Sigma$ and $\Omega \backslash \Sigma$, by setting
$$
     \Sigma = \left\{ x : u(x) > \alpha \right\} \mbox{ for a.e. } \alpha \in (0,1),
$$
where $u$ is the solution of problem~\eqref{eq:cen_continuous}.
Note that~\eqref{eq:cen_continuous} is the result of a suitable relaxation of the Chan-Vese model~\cite{bib:ChanVese2001}
leading to a convex formulation of that problem for any given~$(c_1, c_2)$.

Here we start from a discrete version of the CEN model. Let
$$
    \Theta =  \left\{ (i,j) : 0 \leq i \leq m-1, \, 0 \leq j \leq n-1\right\}
$$
be a discretization of $\Omega$ consisting of an $ m \times n$ grid of pixels
and let
$$
    | \nabla_x u |_{i,j}  = | \delta_x^+ u |_{i,j} , \quad  | \nabla_y u |_{i,j} = | \delta_y^+ u |_{i,j}
$$ 
where $\delta_x^+$  and $\delta_y^+$ are the forward finite-difference operators in the $x$- and $y$-directions, with unit spacing,
and the values $u_{i,j}$ with indices outside~$\Theta$ are defined by replication. We consider the following discrete version
of~\eqref{eq:cen_continuous} with anisotropic discrete total variation (TV):
\begin{equation} \label{eq:cen_discr}
\begin{array}{ll}
    \displaystyle \underset{u, \, c_1, \, c_2}{\minimize} &  \displaystyle \sum_{i,j} \left( | \nabla_x u |_{i,j} + | \nabla_y u |_{i,j} \right) \, + \,
                         \lambda \sum_{i,j} \left( \left( c_1-\bar u_{i,j} \right)^2 u_{i,j} + \left( c_2-\bar u_{i,j} \right)^2 (1-u_{i,j})  \right), \\
   \st             & 0 \le u_{i,j} \le 1, \\[2pt]  
                    & c_1, c_2 > 0.
\end{array}
\end{equation}
After the minimization problem~\eqref{eq:cen_discr} is solved, the segmentation is obtained
by taking
\begin{equation}
\label{eq:sigma_discr}
    \Sigma =  \left  \{(i,j) \in \Theta : u_{i,j} > \alpha \right \},
\end{equation}
for some $\alpha \in (0,1)$. Problem~\eqref{eq:cen_discr} is usually solved by alternating the minimization with
respect to $u$ and the minimization with respect to $c_1$ and $c_2$. In the sequel, we denote $E(u,c_1,c_2)$
the objective function in~\eqref{eq:cen_discr}, and $F(u)$ and $G(u, c_1, c_2)$ its discrete regularization and fidelity
terms, respectively.

%

The selection of a parameter $\lambda$ able to balance $F(u)$ and $G(u, c_1, c_2)$ and to produce a meaningful solution
is a critical issue. Too large values of $\lambda$ may produce oversegmentation,
while too small values of $\lambda$ may produce undersegmentation~\cite{bib:Candemir2010}.
Furthermore, a constant value of $\lambda$ may not be suitable for the whole image, i.e., different regions of the image may
need different values. In recent years, spatially adaptive techniques have been  proposed, focusing on local information, such as gradient,
curvature, texture and noise estimation cues -- see, e.g.,~\cite{bib:Rao2010}. Space-variant regularization
has been also widely investigated in the context of image restoration, using TV and its
generalizations -- see, e.g.,~\cite{bib:Fong2011,bib:Bredies2013,bib:ChanRH2013,bib:Ma2018,bib:CalatroniLanzaPragliolaSgallari2019}.

In this work, we propose some techniques for setting the parameter $\lambda$ in an adaptive way based on
spatial information, in order to prevent excessive regularization of smooth regions while preserving spatial features
in nonsmooth areas. Our techniques are based on the so-called cartoon-texture decomposition of the given image,
on the mean and median filters, and on thresholding techniques.
The resulting locally adaptive segmentation model can be solved either by smoothing the discrete TV
term -- see, e.g.,~\cite{bib:Antonelli2018,bib:diserafino2020} -- and applying optimization solvers for differentiable problems
such as spectral gradient methods~\cite{bib:BirginMartinezRaydan2000,bib:Antonelli2016SPGA,bib:diserafino2018,bib:crisci2020}
or by using directly optimization solvers for nondifferentiable problems, such as Bregman, proximal and ADMM methods
\cite{bib:Bregman1967,bib:Osher2005mms,bib:Goldstein2009sb,bib:Campagna2017,bib:Parikh2014,bib:Beck2009TV,bib:Boyd2011admm}.
In this work we use an alternating minimization procedure exploiting the split Bregman (SB) method proposed in~\cite{bib:Goldstein2009sb}.
The results of numerical experiments on images with different characteristics show the effectiveness of our approach
and the advantages coming from using local regularization.

The rest of this paper is organized as follows. In Section~\ref{sec:locally-adptive-lambda} we propose our spatially adaptive techniques.
In Section~\ref{sec:split-bregman} we describe the solution of the segmentation model using those techniques by the aforementioned
SB-based alternating minimization method. In Section~\ref{sec:results} we discuss the results obtained with our approaches on several
test images, performing also a comparison with the CEN model and a segmentation model developed for images with texture.
The results show the effectiveness of our approach and the advantages coming from the use of local regularization.
Some conclusions are provided in Section~\ref{sec:conclusions}.

\section{Defining local regularization by exploiting spatial information\label{sec:locally-adptive-lambda}}

The regularization parameter $ \lambda $ plays an important role in the segmentation process, because it controls
the tradeoff between data fidelity and regularization. In general, the smaller the parameter in~\eqref{eq:cen_discr}
the smoother the image content, i.e., image details are flattened or blurred. Conversely, the larger the parameter the more the enhancement
of image details, and hence noise may be retained or amplified. Therefore, $\lambda$ should be selected according to local spatial information.
A small value of $\lambda$ should be used in the smooth regions of the image to suppress noise, while a large value of $\lambda$
should be considered to preserve spatial features in the nonsmooth regions. In other words, a matrix $\Lambda = (\lambda_{i,j})$ should
be associated with the image, where $\lambda_{i,j}$ weighs pixel $(i,j)$, as follows:
\begin{equation} \label{eq:cen_discr_lambdaij}
\begin{array}{ll}
    \displaystyle \underset{u, c_1, c_2}{\minimize} &  \displaystyle \sum_{i,j} \left( | \nabla_x u |_{i,j} + | \nabla_y u |_{i,j} \right) \, + \,
                         \sum_{i,j} \lambda_{i,j}  \left( \left( c_1-\bar u_{i,j} \right)^2 u_{i,j} + \left( c_2-\bar u_{i,j} \right)^2 (1-u_{i,j})  \right), \\
    \st           & 0 \le u_{i,j} \le 1, \\[2pt]  
                   & c_1, c_2 > 0.
\end{array}
\end{equation}

Furthermore, in order to avoid oversegmentation or undersegmentation, it is convenient to fix
a minimum and a maximum value for the entries of $\Lambda$, so as to drive the level of regularization in a reasonable range,
depending on the image to be segmented.

We define $\Lambda$ as a function of the image $\bar u$ to be segmented:
\begin{equation}
\label{eq:fu} 
f: \bar u_{i,j}  \rightarrow \lambda_{i,j} \in [\lambda_{min},\lambda_{max}], 
\end{equation}
where $0 < \lambda_{min} < \lambda_{max} < \infty$. We propose three choices of $f$, detailed in the next subsections.

We note that problem~\eqref{eq:cen_discr_lambdaij} still has a unique solution for any fixed $(c_1, c_2)$, as a consequence of the
next proposition.
\begin{prop}
\label{th:convexity}
For any fixed $(c_1, c_2)$, problem~\eqref{eq:cen_discr_lambdaij} is a convex problem.
\end{prop}
\begin{proof}
Since the CEN model is convex for any fixed $(c_1, c_2)$, the thesis immediately follows
from the fact that the parameters $\lambda_{i,j}$ are constant with respect to $u$.
\end{proof}

\subsection{Regularization based on the cartoon-texture decomposition\label{cartoon-texture}}

We define $f(\bar u_{i,j})$ by using the Cartoon-Texture Decomposition (CTD)
of the image discussed in~\cite{bib:Buades2010CTD}. CTD splits any image $u$ into the sum of two images, $w$ and $v$,
such that $w$ represents the cartoon or geometric (piecewise-smooth) component of $u$,
while $v$ represents the oscillatory or textured component, i.e., $v$ contains essentially textures,
oscillating patterns, fine details and noise. The algorithm for computing CTD acts as described next.
For each image pixel, we decide whether it belongs to the cartoon or the textural part by computing
a local indicator associated with an image window around the pixel. The main feature of a textured region
is its high TV, which decreases very fast under low-pass filtering. This leads to the following definition
of local total variation (LTV) at a pixel $(i,j)$:
\begin{equation}
\label{eq:ltv}
LTV_\sigma (u)_{i,j} = \left( L_\sigma * | \nabla u |_{i,j} \right),
\end{equation}
\noindent where $L_\sigma$ is a low-pass filter, $\sigma$ is a scale parameter,
$ | \nabla u |_{i,j} = \sqrt{ (\nabla_x u)_{i,j}^2 +  (\nabla_y u)_{i,j}^2 } $ and $*$~denotes the convolution product.
The relative reduction rate of LTV,
\begin{equation}
\label{eq:rrrltv}
(\rho_\sigma)_{i,j} =\frac{LTV_\sigma (u)_{i,j}-LTV_\sigma (L_\sigma *u)_{i,j}}{LTV_\sigma (u)_{i,j}} ,
\end{equation}
gives the local oscillatory behavior of the image.
A value of $(\rho_\sigma)_{i,j}$ close to $0$ means that there is little relative reduction of LTV by the low-pass filter,
thus the pixel $(i,j)$ belongs to the cartoon. Conversely, $(\rho_\sigma)_{i,j}$ close to 1 indicates that the relative reduction
is large and hence the pixel belongs to a textured region.

We use~\eqref{eq:rrrltv} for defining the weights $\lambda_{i,j}$. The basic idea is that a large
regularization parameter is needed if the pixel $(i,j)$  belongs to the cartoon, while the parameter must be reduced 
in texture regions. Therefore, we set the function $f$ in~\eqref{eq:fu} as
\begin{equation}
\label{eq:ctd}
f(\bar u_{i,j}) \equiv f^{\text{CTD}}(\bar u_{i,j}) = \max \left \{ \frac{\lambda_{min}}{\lambda_{max}}, \, 1-(\rho_\sigma)_{i,j}  \right \} \lambda_{max}
\end{equation}
and $(\rho_\sigma)_{i,j}$ is defined by using the given image $\bar u$. Following~\cite{bib:Buades2011CTDimplem}, we set $L_\sigma$ equal to the
Gaussian filter.

\subsection{Regularization based on the mean and median filters\label{sec:filters}}

We define a technique based on spatial filters that are commonly used to enhance low-frequency details or
to preserve edges~\cite{bib:jain1989IP,bib:Gonzalez2001IP}.
More precisely, we combine the mean and median filters; the former aims at identifying smooth regions,
where the regularization parameter can take small values, the latter aims at identifying edges, where the parameter
must remain large.
Mean filtering is a simple and easy-to-implement method for smoothing images, i.e., for reducing the intensity
variation between a pixel and its neighbors. It also removes high-frequencies components due to the noise and the
edges in the image, so the mean filter is a low-pass filter. The median filter preserves edges and useful details in the image.

Based on these considerations, we define a weight function as follows:
\begin{equation} \label{eq:MMweights}
  \omega_{i,j} = \left\{ \begin{array}{cl}
              | \bar u_{i,j} - (L_{h1} * \bar u)_{i,j}| & \mbox{ if } |(L_{h1} * \bar u)_{i,j} - (M_{h2} * \bar u)_{i,j}| < t, \\[4pt]
                 1                                                     & \mbox{ otherwise.} \\
              \end{array} \right.
\end{equation}
where $h_1$ is the window size of the mean filter $L_{h1}$, $h_2$ is the window size of the median filter $M_{h2}$,
and $t$ is a threshold value acting as a cutoff between the two filters. Note that $0 \leq \omega_{i,j} \leq 1$
and the pixels in homogeneous regions have $\omega_{i,j}$ close to $1$. The function $f$ in~\eqref{eq:fu} is set as
\begin{equation}\label{eq:mm}
f(\bar u_{i,j}) \equiv f^{\text{MM}}(\bar u_{i,j}) = \max \left \{ \frac{\lambda_{min}}{\lambda_{max}}, \,1-\omega_{i,j}  \right \} \lambda_{max}, 
\end{equation}
where MM stands for ``Mean and Median''.

\subsection{Regularization based on thresholding\label{sec:thr}}

This approach implicitly exploits the definition of $\Sigma$ in~\eqref{eq:sigma_discr} to set $\lambda_{i,j}$. The idea is to use
large values of $\lambda_{i,j}$ when $u_{i,j}$ is close to 1 and small values when $u_{i,j} $ is close to 0.
Therefore, the parameters $\lambda_{i,j}$ are not defined in terms of the given image $\bar u$ only. If the function $u$ identifying
the segmentation were known a priori, we could define $\lambda_{i,j} = f(u_{i,j})$
as follows:
\begin{equation}
\label{eq:thr}
f(u_{i,j}) \equiv f^{\text{THR}}(u_{i,j}) = 10^{\eta_{i,j}} \in [\lambda_{min},\lambda_{max}],
\end{equation}
where
\begin{equation*}
\eta_{i,j} = emax - (1 - u_{i,j}) (emax - emin),
\end{equation*}
$emax = \log_{10}\lambda_{max}$ and $emin = \log_{10}\lambda_{min}$.

Since the function $u$ must be computed by minimizing~\eqref{eq:cen_discr} and this is done by using
an iterative procedure, we decided to update $\lambda_{i,j}$ at each iteration, using the current value of $u_{i,j}$. 
On the other hand, in this case evaluating $f$ is computationally cheaper than in the previous approaches, which apply
two-dimensional convolution operators; thus the cost of the iterative update of $\lambda_{i,j}$ is practically negligible.

\section{Solution by split Bregman iterations\label{sec:split-bregman}}

As mentioned in Section~\ref{sec:intro}, we use an alternating minimization method to solve problem~\eqref{eq:cen_discr_lambdaij}.
Given $u_{i,j}^{k-1}$, by imposing the first-order optimality conditions with respect to $c_1$ and $c_2$, we get
\begin{equation}
\label{eq:loc_c1c2}
c_1^k=\frac{\sum_{i,j} \lambda_{i,j}\bar u_{i,j} u_{i,j}^{k-1}} {\sum_{i,j} \lambda_{i,j} u_{i,j}^{k-1}},
\quad
c_2^k=\frac{\sum_{i,j} \lambda_{i,j}\bar u_{i,j} (1-u_{i,j}^{k-1})} {\sum_{i,j} \lambda_{i,j}(1-u_{i,j}^{k-1})} .
\end{equation}
For the solution of~\eqref{eq:cen_discr_lambdaij} with respect to $u$ we use the split Bregman (SB) method~\cite{bib:Goldstein2009sb}.
Let
\begin{equation} \label{eq:loc_fidelity}
G_{loc} (u, c_1, c_2) =  \sum_{i,j} \lambda_{i,j}  \left( \left( c_1-\bar u_{i,j} \right)^2 - \left( c_2-\bar u_{i,j} \right)^2  \right) u_{i,j} \\
                                  =  \sum_{i,j} r_{i,j} u_{i,j} ,
\end{equation}
where 
$$
    r_{i,j} = \lambda_{i,j}  \left( \left( c_1-\bar u_{i,j} \right)^2 - \left( c_2-\bar u_{i,j} \right)^2  \right) .
$$
%
Following~\cite{bib:GoldsteinBressonOsher2010}, we reformulate the minimization problem as follows:
\begin{equation}
\label{eq:sb_model}
\begin{array}{cl}
\displaystyle \underset{u,d_x,d_y}{\minimize}  & \displaystyle \| d_x \|_1 + \| d_y \|_1 + G_{loc}(u, c_1, c_2), \\
\mbox{s.t.} & 0 \le u_{i,j} \le 1, \\
                  & d_x=\nabla_x u ,\\
                  & d_y=\nabla_y u.\\  
\end{array}
\end{equation}
Given $c_1^k$ and $c_2^k$, the SB method applied to~\eqref{eq:sb_model} reads:
\begin{align}
    u^k      & = \underset{0 \leq u_{i,j} \leq 1}{\argmin} \,G_{loc}(u, c_1^k,c_2^k) + \frac{\mu}{2} \| d_x^{k-1} - \nabla_x u - b_x^{k-1} \|_2^2
                                  + \frac{\mu}{2} \| d_y^{k-1} - \nabla_y u - b_y^{k-1} \|^2_2 , \nonumber \\
    d_x^k  & = \underset{d_x}{\argmin} \, \| d_x \|_1 + \frac{\mu}{2} \| d_x - \nabla_x u^k - b_x^{k-1} \|_2^2 ,  \nonumber \\
    d_x^k  & = \underset{d_y}{\argmin} \, \| d_y \|_1 + \frac{\mu}{2} \| d_y - \nabla_y u^k - b_y^{k-1} \|_2^2 ,  \label{eq:ASBSegloc} \\
    b_x^k  & = b^{k-1}_x + \mu(\nabla_x u^k - d_x^k) ,  \nonumber \\[1mm]
    b_y^k  & = b^{k-1}_y + \mu(\nabla_x u^k - d_y^k) ,  \nonumber
\end{align}
where $ \mu > 0 $.

Closed-form solutions of the minimization problems with respect to $d_x$ and $d_y$
can be computed using the soft-thresholding operator:
$$
     d_x^k = \mathcal{S} \left( \nabla_x  u^k+b_x^k, \frac{1}{\mu} \right), \quad d_y^k = \mathcal{S} \left( \nabla_y  u^k+b_y^k, \frac{1}{\mu} \right),
$$
where, for any $v = ( v_{i,j} )$ and any scalar $\gamma > 0$,
$$
     \mathcal{S} (v, \gamma) = z = (z_{i,j}),  \quad  z_{i,j} = \frac{z_{i,j}}{| z_{i,j} |} \max \left\{ | z_{i,j} | - \gamma, \, 0 \right\}.
$$

Finally, an approximate solution to the minimization problem with respect to $u$ can be obtained by applying Gauss-Seidel (GS) iterations
to the following system, as explained in~~\cite{bib:GoldsteinBressonOsher2010}:
\begin{equation}
\label{eq:sys} 
- \Delta u_{i,j} = \frac{r_{i,j}}{\mu} + ( \nabla_x ( b_x^k - d_x^k)_{i,j} )+ ( \nabla_y ( b_y^k - d_y^k)_{i,j} ),
\end{equation}
where $\Delta$ is the classical finite-difference discretization of the Laplacian.
If the solution to~\eqref{eq:sys} lies outside $[0, 1]^{m \times n}$, then it is projected onto that set. We denote $\mathcal{P}_{[0,1]}$
the corresponding projection operator.

The overall solution method is outlined in Algorithm~\ref{alg:SBLOC}. Note that when the approach in Section~\ref{sec:thr} is used,
the values $\lambda_{i,j}$ must be updated at each iteration $k$, using~\eqref{eq:thr} with $u = u^k$.

\begin{algorithm}[h!] 
\caption{SB-based method for spatially adaptive segmentation\label{alg:SBLOC}}
\SetAlgoLined
{\setstretch{1.2}
    \SetKwInOut{KwIn}{Input} \KwIn{ $\bar{u}, \, \lambda_{min}, \, \lambda_{max}, \, f, \, \mu, \, \alpha$ (with $f$
    defined in~\eqref{eq:ctd} or \eqref{eq:mm} or \eqref{eq:thr})}
    \SetKwInOut{KwOut}{Output} \KwOut{ $u, \, c_1, \, c_2$}
    Set  $u^0=\bar u, \, d_x^0=0, \, d_y^0=0, \, b_x^0=0, \, b_y^0=0$ \\
    Compute  $\Lambda = f(u^0)$  \\
    \For{ $k = 1, 2, \ldots$ }{
       Compute $c_1^k$ and $c_2^k$ by \eqref{eq:loc_c1c2}  \\	
       Compute $u^k$ by applying GS iterations to system~\eqref{eq:sys}  \\
       $u^k  = \mathcal{P}_{[0,1]}(u^k)$  \\
       $d_x^k = \mathcal{S}(\nabla_x  u^k + b_x^k, 1/\mu) $   \\
       $d_y^k = \mathcal{S}(\nabla_y  u^k + b_y^k, 1/\mu) $   \\
       Update $ b_x^k $ and $ b_y^k$ according to \eqref{eq:ASBSegloc} \\
    }
    Set $\Sigma = \left \{ (i,j) \in \Theta : u_{i,j}^k > \alpha \right \}$    \\
} 
\end{algorithm}

\section{Results and comparisons\label{sec:results}}

The three spatially adaptive regularization techniques were implemented in MATLAB, using the Image Processing Toolbox.
Algorithm~\ref{alg:SBLOC} was implemented by modifying the \emph{Fast Global Minimization Algorithm for Active Contour Models}
by X.~Bresson, available from \texttt{http://htmlpreview.github.io/?https://github.com/xbresson/old\_codes/blob/master/codes.html}.
This is a C code with a MATLAB MEX interface.

$L_\sigma$ in~\eqref{eq:ltv} is defined as a rotationally symmetric Gaussian filter with size 3 and standard deviation
$\sigma=2$. The mean and median filters use windows of size 3 and 7, respectively, and the parameter $t$ in~\eqref{eq:MMweights}
is set as $t=0.5$. The parameter $\alpha = 0.5$ is used to identify the domain $\Sigma$ according to~\eqref{eq:sigma_discr}.

In the original and modified codes, the SB iterations are stopped as follows:   
\begin{equation}\label{eq:stop-crit}
| \mathtt{diff}^k - \mathtt{diff}^{k-1} | \le \mathtt{tol} \quad  \mbox{and} \quad k > \mathtt{maxit} ,
\end{equation}
where
$$
     \mathtt{diff}^l = \frac{\mathtt{sd}(u^l)}{\mathtt{sd}(u^l) \cdot \mathtt{sd}(u^{l-1})}, \quad
     \mathtt{sd}(u^l) = \sum_{i,j} (u_{i,j}^l - u_{i,j}^l)^2, 
$$
$\mathtt{tol}$ is a given tolerance, and $\mathtt{maxit}$ denote the maximum number of outer iterations.
The stopping criterion for the GS iterations is
$$
     E(u^k) = 1 - \frac{|\mathtt{msd}^k - \mathtt{msd}^1|}{\mathtt{msd}^1} \le \mathtt{tol_{GS}}
     \quad  \mbox{and} \quad k > \mathtt{maxit_{GS}} 
$$
where
$$
    \mathtt{msd}^l =  \frac{1}{m n} \sum_{i,j} (u_{i,j}^l - u_{i,j}^{l-1})^2,
$$
and $\mathtt{tol_{GS}}$ and $\mathtt{maxit_{GS}}$ are the tolerance and the maximum number of iterations
for the GS method, respectively.
In our experiments we set $\mathtt{tol} = 10^{-6}$, $\mathtt{maxit} = 30$, $\mathtt{tol_{GS}} = 10^{-2}$ and
$\mathtt{maxit_{GS}} = 50$.

\begin{table}[b!]
\begin{center}
\caption{Test images and their sizes.\label{tab:images}}
\setlength{\tabcolsep}{6pt}
\begin{tabular}{lc}
\toprule
\textbf{image}             & \textbf{size (pixels)} \\
\midrule
\texttt{bacteria}	          &  $233 \times 256$ \\
\texttt{bacteria2} 	  & $380 \times 391$ \\
\texttt{brain}		  & $210 \times 210$ \\
\texttt{cameraman}     & $204 \times 204$ \\
\texttt{flowerbed}         & $321 \times 481$ \\
\texttt{ninetyeight}	  & $300 \times 225$ \\
\texttt{squirrel}	          & $167 \times 230$ \\
\texttt{tiger}                 & $321 \times 481$ \\
\bottomrule
\end{tabular}
\end{center}
\end{table}
The adaptive models are compared with the original CEN model on different images widely used in image processing tests,
listed in Table~\ref{tab:images} and shown in Figures~\ref{fig:CENvsCTD} to~\ref{fig:smoothVStexture}.
In particular, the images \texttt{bacteria}, \texttt{bacteria2}, \texttt{brain}, \texttt{cameraman}, \texttt{squirrel}
and \texttt{tiger} are included in Bresson's code distribution, \texttt{flower\-bed} has been downloaded from
the Berkeley segmentation dataset~\cite{bib:Arbelaez2011BK500} available from 
\texttt{https://www2.eecs.berkeley.edu/Research/Projects/CS/vision/grouping/resources.html} (image \texttt{\#86016}),
and \texttt{ninetyeight} is available from \texttt{https://tineye.com/query/2817cf0d186fbfe263a188952829a3b5e699d839}.
We note that \texttt{tiger} is included in Bresson's code as a test problem for a segmentation model specifically developed for
images with texture~\cite{bib:HouhouThiranBresson2009}, which is also implemented in the code. This model uses the well-known
Kullback-Leibler divergence function regularized with a TV term. The model is solved with the SB method.
We perform the segmentation of \texttt{tiger} with the CEN model, the textural segmentation model and our approaches,
to investigate whether our locally adaptive model can be also suitable for textural images.
The \texttt{cameraman} image is perturbed with
additive Gaussian noise, with zero mean and two values of standard deviation, $\sigma = 15, 25$, with the aim of evaluating
the behavior of our adaptive approaches on noisy images. The noisy images are called \texttt{cameraman15} and \texttt{cameraman25}.

For the images provided with Bresson's code, the values of $\lambda$ and $\mu$ associated with the original CEN model are set as in that code.
For the remaining images, the values of $\lambda$ and $\mu$ are set by trial and error, following the empirical rule reported in Bresson's code.
The values of $\lambda_{min}$ and $\lambda_{max}$ used in the spatially adaptive approaches are chosen so that the corresponding
$\lambda$ in the original CEN model is in $[\lambda_{min}, \lambda_{max}]$, with few exceptions.
The associated values of $\mu$ are set as in the non-adaptive case. The values of $\lambda$, $\lambda_{min}$, $\lambda_{max}$ and $\mu$ 
used for each image are specified in Tables~\ref{tab:params_and_iters} to~\ref{tab:cameraman}.
The values of $\lambda_{i,j}$ in the adaptive models are initialized as specified in~\eqref{eq:ctd}, \eqref{eq:mm} and \eqref{eq:thr}.
As described in Section~\ref{sec:thr}, in the strategy based on thresholding those values change at each iteration
$k$ of Algorithm~\ref{alg:SBLOC}. It is worth noting that our approach also simplifies the choice of the regularization parameter,
which may be a time-consuming task.

The initial approximation $u^0$ is set equal to the image $\bar u_{i,j}$, which takes its values in $[0,1]$, as specified in
Section~\ref{sec:intro}. This is used to compute the starting values of $c_1$ and~$c_2$, and
$ s_{i,j} = \left( c_1-\bar u_{i,j} \right)^2 - \left( c_2-\bar u_{i,j} \right)^2 $ for all $(i,j)$. In the original non-adaptive code,
the value of $\lambda$ is scaled using the difference between the maximum and the minimum value of $s_{i,j}$;
the same scaling is applied to $\lambda_{min}$ and $\lambda_{max}$ in the implementations of the adaptive approaches.

We run the tests on an Intel Core~i7 processor with clock frequency of~2.6 GHz, 8~GB of RAM, and a 64-bit Linux system.

For each of the six images \texttt{bacteria}, \texttt{bacteria2}, \texttt{brain}, \texttt{flowerbed}, \texttt{ninetyeight} and \texttt{squirrel}
we show the results obtained with the spatially adaptive strategy yielding the best segmentation for that image. The corresponding
(unscaled) values of $\lambda$, $\lambda_{min}$, $\lambda_{max}$, the value of $\mu$, as well as
the number of outer iterations and the mean number of GS iterations per outer iteration are reported in Table~\ref{tab:params_and_iters}.
Note that for \texttt{squirrel} we use two values of $\lambda_{min}$, one equal to the value of $\lambda$ in the CEN model
and the other greater than that, obtained by trial and error. Both values of $\lambda_{min}$ produce the same segmentation,
but the larger value of $\lambda_{max}$ reduces the number of outer and GS iterations.
The segmentations corresponding to the data in Table~\ref{tab:params_and_iters} are shown in Figures~\ref{fig:CENvsCTD} to~\ref{fig:CENvsTHR}.
For \texttt{squirrel} we display the CTD segmentation computed by using the larger value of $\lambda_{min}$.

\begin{table}[h!]
\begin{center}
\caption{Segmentation of six test images by the CEN and spatially adaptive models: parameters and iterations.
The value of $\mu$ is the same for all the models, thus it is reported only once per image.
\label{tab:params_and_iters}}
\setlength{\tabcolsep}{6pt}
\begin{tabular}{ lccrlccrl}
\toprule
                                    & $\lambda$           & $\mu$                   & $\mathtt{it}$  & $\mathtt{it_{GS}}$ 
                                    & $\lambda_{min}$ & $\lambda_{max}$ & $\mathtt {it}$ & $\mathtt{it_{GS}}$ \\
\midrule
\textbf{image}             & \multicolumn{4}{c}{\textbf{CEN}}  & \multicolumn{4}{c}{\textbf{CTD}}  \\
\midrule
\texttt{brain}                 & $0.7$e$+3$ & $0.1$e$+4$  & 5 & 19.2 & $0.9$e$+3$ & $0.4$e$+5$ & 5 & 14.8  \\ 
\texttt{squirrel}             &  0.118e+3    & $0.1$e$+4$  & 6 & 50    & $0.118$e$+3$ & $0.1$e$+4$ & 15 & 50  \\
                                    &                     &                      &    &         & $0.2$e$+3$ & $0.1$e$+4$ & 5 & 45  \\
\midrule
                                    & \multicolumn{4}{c}{\textbf{CEN}}  & \multicolumn{4}{c}{\textbf{MM}} \\
\midrule
\texttt{bacteria2}          & $0.1$e$+3$ & $0.1$e$+4$  & 2 & 50   & $0.1$e$+3$ & $0.1$e$+4$ & 10 & 50 \\
\texttt{ninetyeight}       & $0.2$e$+2$ & $0.1$e$+3$  & 4 & 38.5 & $0.1$e$+2$ & $0.5$e$+2$ &  3 & 50 \\
\midrule
                                    & \multicolumn{4}{c}{\textbf{CEN}} & \multicolumn{4}{c}{\textbf{THR}} \\
\midrule
\texttt{bacteria}            & $0.5$e$+4$ & $0.1$e$+5$ & 6 & 27.5  & $0.4$e$+4$  & $0.8$e$+4$ & 7 & 30.7 \\
\texttt{flowerbed}         & $0.1$e$+2$ & $0.1$e$+4$ & 2 & 27     & $0.3$e$+2$         & $0.1$e$+3$         & 2 & 16.5 \\ 
\bottomrule
\end{tabular}
\end{center}
\end{table}

We see that on the selected problems CTD reduces the number of outer and GS iterations with respect to the CEN model; on the other hand,
the setup of the regularization parameters is computationally more expensive. In terms of iterations, there is no clear winner between CEN and MM
and between CEN and THR. The models based on the spatially adaptive techniques are slightly more expensive than the CEN model in this case too.
A~significant result is that the segmentations obtained with the adaptive techniques appear better than those obtained with the non-adaptive model,
i.e., the spatially adaptive models can better outline boundaries between objects and foreground. This is clearly visible by looking
at the segmentations of \texttt{brain}, \texttt{bacteria2} (see the upper right corner), \texttt{ninetyeight}, \texttt{bacteria} (see the shape of the bacteria).
It is also worth noting that the adaptive model based on THR removes textural details that do not belong to the flowerbed in the homonym image.

The latter observation is confirmed by the experiments performed on~\texttt{tiger}. The corresponding model and algorithmic details are reported in
Table~\ref{tab:tiger}, while the segmentations are shown in Figure~\ref{fig:smoothVStexture} along with (visual) information on quantities used
to define $\lambda_{i,j}$ (see~\eqref{eq:rrrltv} and~\eqref{eq:MMweights}). We see that the CTD and MM strategies produce satisfactory
results, although they have been obtained by generalizing a non-textural model.

Finally, we show the results obtained on \texttt{cameraman} and its noisy versions by using CEN, CTD, MM, and THR.
The methods perform comparable numbers of inner and GS iterations (see Table~\ref{tab:cameraman}), but the spatially adaptive
model THR yields some improvement over the CEN model on the noisy images (Figure~\ref{fig:noisy}).

\section{Conclusions\label{sec:conclusions}}

We introduced spatially adaptive regularization in a well-established variational segmentation model with
the aim of improving the segmentation of images by suitably taking into account their smooth and nonsmooth regions.
To this aim, we introduced three techniques, based on the application of suitable spatial filters or thresholding.
The locally adaptive models, solved via an alternating minimization method using split Bregman iterations, showed
the effectiveness of our approaches on several images, including also textural and noisy images. We also believe that
the proposed models may simplify the setup of the regularization parameter.

Future work can include the extension of our spatially adaptive strategies to other segmentation models and the
development of further adaptive techniques.

\begin{figure*}[ht!]
\medskip
\begin{center}
\newcolumntype{C}{>{\centering\arraybackslash} m{.30\textwidth} }
\begin{tabular}{CCC}
\includegraphics[width=.24\textwidth]{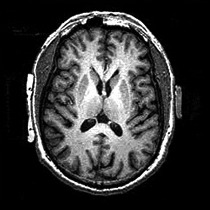} &
\includegraphics[width=.24\textwidth]{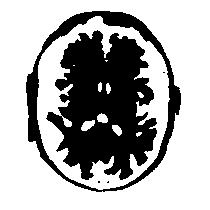} &
\includegraphics[width=.24\textwidth]{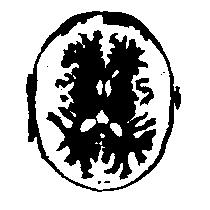} \\[1mm]
\multicolumn{1}{c}{\small \texttt{brain}}  &
\multicolumn{1}{c}{\small CEN }  &
\multicolumn{1}{c}{\small CTD }\\[1mm]
\\[1mm]
\includegraphics[width=.24\textwidth]{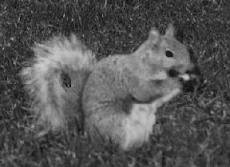} &
\includegraphics[width=.24\textwidth]{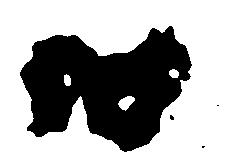} &
\includegraphics[width=.24\textwidth]{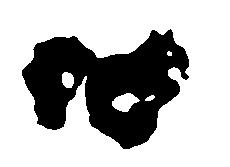} \\[1mm]
\multicolumn{1}{c}{\small \texttt{squirrel} }  &
\multicolumn{1}{c}{\small CEN }  &
\multicolumn{1}{c}{\small CTD  }\\[1mm]
\end{tabular}
\vskip 1pt
\caption{Segmentations of \texttt{brain} and \texttt{squirrel} by using CEN and CTD.
\label{fig:CENvsCTD}}
\end{center}
\end{figure*}
\begin{figure}[h!]
\medskip
\begin{center}
\newcolumntype{C}{>{\centering\arraybackslash} m{.30\textwidth} }
\begin{tabular}{CCC}
\includegraphics[width=.24\textwidth]{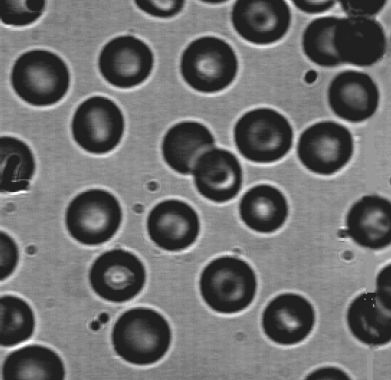} &
\includegraphics[width=.24\textwidth]{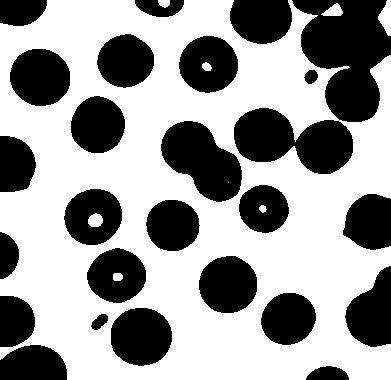} &
\includegraphics[width=.24\textwidth]{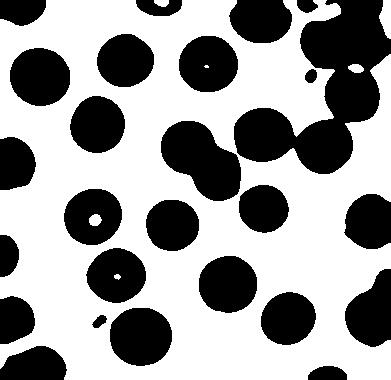} \\[1mm]
\multicolumn{1}{c}{\small \texttt{bacteria2}}  &
\multicolumn{1}{c}{\small CEN }  &
\multicolumn{1}{c}{\small MM  }\\[1mm]
\\[1mm]
\includegraphics[width=.24\textwidth]{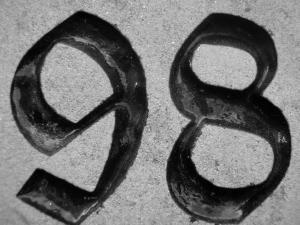} &
\includegraphics[width=.24\textwidth]{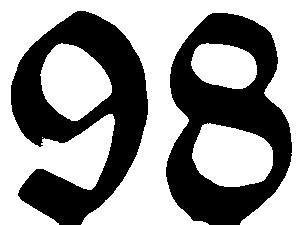} &
\includegraphics[width=.24\textwidth]{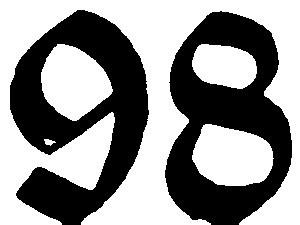} \\[1mm]
\multicolumn{1}{c}{\small \texttt{ninetyeight}}  &
\multicolumn{1}{c}{\small CEN }  &
\multicolumn{1}{c}{\small MM } \\[1mm]
\end{tabular}
\vskip 1mm
\caption{Segmentations of \texttt{bacteria2} and \texttt{ninetyeight} by using CEN and MM.
\label{fig:CENvsMM}}
\end{center}
\end{figure}

\begin{figure}[ht!]
\medskip
\begin{center}
\newcolumntype{C}{>{\centering\arraybackslash} m{.30\textwidth} }
\begin{tabular}{CCC}
\includegraphics[width=.24\textwidth]{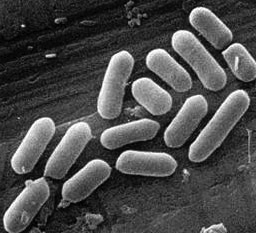} &
\includegraphics[width=.24\textwidth]{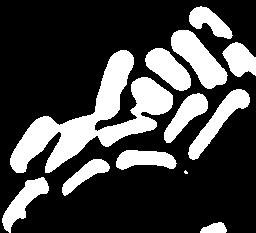} &
\includegraphics[width=.24\textwidth]{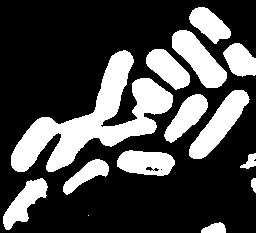} \\[1mm]
\multicolumn{1}{c}{\small \texttt{bacteria}}  &
\multicolumn{1}{c}{\small CEN }  &
\multicolumn{1}{c}{\small THR } \\[1mm]
\\[1mm]
\includegraphics[width=.24\textwidth]{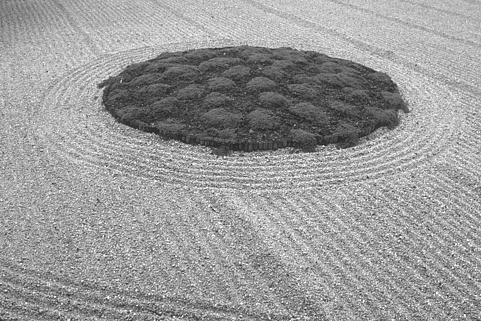} &
\includegraphics[width=.24\textwidth]{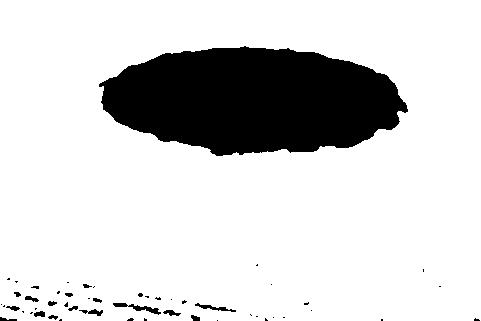} &
\includegraphics[width=.24\textwidth]{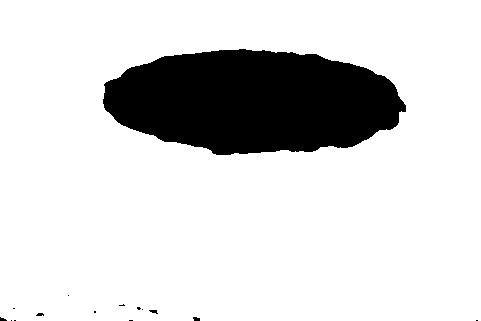} \\[1mm]
\multicolumn{1}{c}{\small \texttt{flowerbed}}  &
\multicolumn{1}{c}{\small CEN }  &
\multicolumn{1}{c}{\small THR } \\[1mm]
\end{tabular}
\vskip 1mm
\caption{Segmentations of \texttt{bacteria} and \texttt{flowerbed} by using CEN and THR.
\label{fig:CENvsTHR}}
\end{center}
\end{figure}

%
\begin{table}[h!]
\begin{center}
\caption{Segmentation of \texttt{tiger} by the texture~\cite{bib:HouhouThiranBresson2009}, CEN, CTD and MM models:
parameters and iterations.\label{tab:tiger}}
\setlength{\tabcolsep}{6pt}
\begin{tabular}{lccrlcccrl}
\toprule
 \textbf{image}             & $\lambda$           & $\mu$                   & $\mathtt{it}$  & $\mathtt{it_{GS}}$
                                    & $\lambda_{min}$ & $\lambda_{max}$ & $\mu$ & $\mathtt {it}$ & $\mathtt{it_{GS}}$ \\
\midrule
                                    & \multicolumn{4}{c}{\textbf{Texture Model}}  & \multicolumn{5}{c}{\textbf{CTD}}  \\
\midrule
\texttt{tiger}                 & $0.3$e$+1$ & $0.1$e$+3$  & 5 & 12 & $0.1$e$+1$ & $0.1$e$+2$  & $0.1$e$+2$  & 8 & 50 \\
\midrule
                                    & \multicolumn{4}{c}{\textbf{CEN}}  & \multicolumn{5}{c}{\textbf{MM}} \\
\midrule
\texttt{tiger}                  & $0.3$e$+2$ & $0.1$e$+3$ & 7 & 49.6  & $0.1$e$+1$ & $0.1$e$+2$ & $0.1$e$+2$  & 8 & 50 \\
\bottomrule
\end{tabular}
\end{center}
\end{table}

\begin{figure}[h!]
\medskip
\begin{center}
\newcolumntype{C}{>{\centering\arraybackslash} m{.30\textwidth} }
\begin{tabular}{CCC}
\includegraphics[width=.24\textwidth]{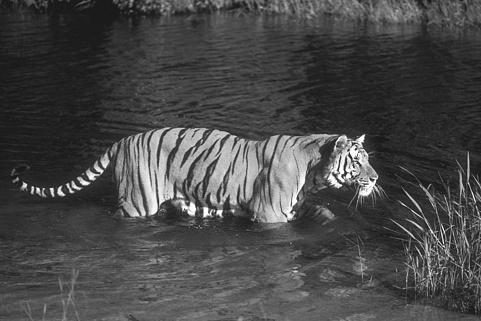} &
\includegraphics[width=.24\textwidth]{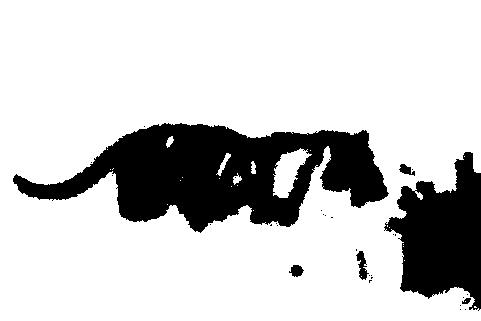} &
\includegraphics[width=.24\textwidth]{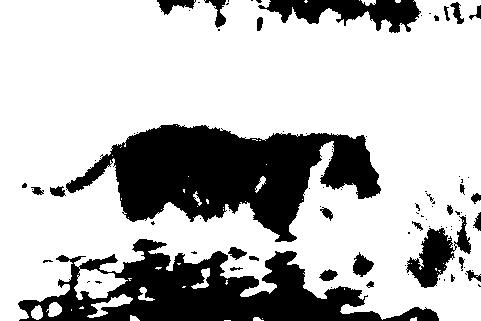} \\[1mm]
\multicolumn{1}{c}{\small \texttt{tiger}}  &
\multicolumn{1}{c}{\small texture model}  &
\multicolumn{1}{c}{\small CEN} \\[1mm]
\\[3mm]
\includegraphics[width=.24\textwidth]{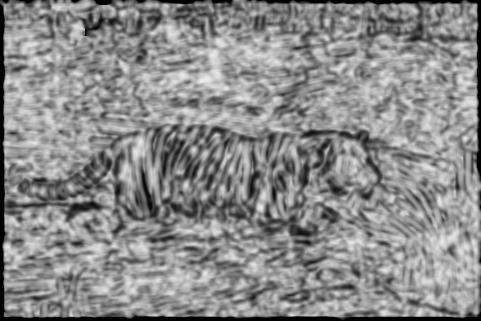} &
\includegraphics[width=.24\textwidth]{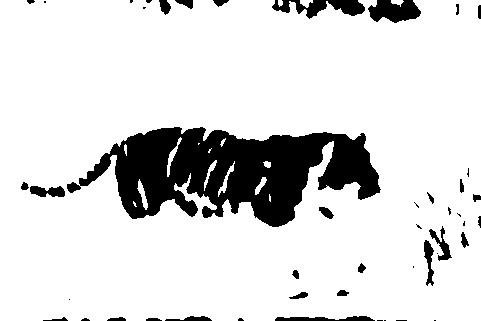} &
 \\[1mm]
 \multicolumn{1}{c}{$\rho_\sigma$}  &
\multicolumn{1}{c}{\small CTD}  &
\multicolumn{1}{c}{} \\[1mm]
\\[1mm]
\includegraphics[width=.24\textwidth]{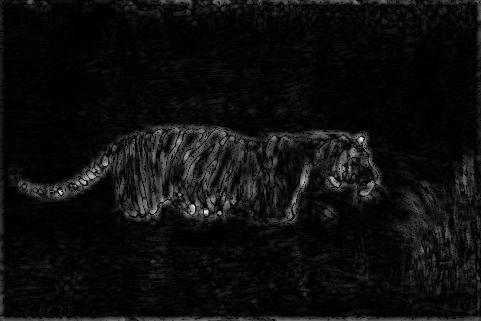} &
\includegraphics[width=.24\textwidth]{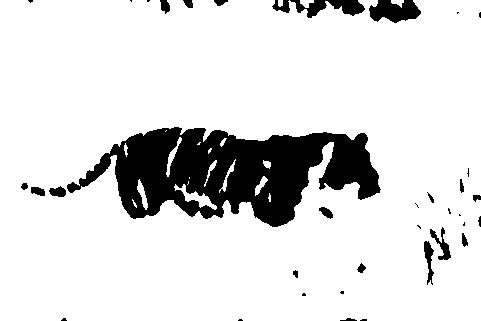} &
\\[1mm]
 \multicolumn{1}{c}{$| L_{h1} * \bar u - M_{h2} * \bar u |$}  &
\multicolumn{1}{c}{\small MM}  &
\multicolumn{1}{c}{} \\[3mm]
\end{tabular}
\vskip 1mm
\caption{Segmentations of \texttt{tiger} obtained by using the textural, CEN, CTD and MM models.
A representation of quantities used to define $\lambda_{i,j}$ (see~\eqref{eq:rrrltv} and~\eqref{eq:MMweights}) is also
provided.\label{fig:smoothVStexture}}
\end{center}
\end{figure}

\begin{table}[h!]
\begin{center}
\caption{Segmentations of {\tt cameraman} and its noisy versions by using CEN and the spatially adaptive models: parameters and iterations.
The value of $\mu$ is the same for all the models, thus it is reported only once per image.\label{tab:cameraman}}
\setlength{\tabcolsep}{6pt}
\begin{tabular}{lccrlccrl}
\toprule
                                               & \multicolumn{4}{c}{\textbf{CEN}}  & \multicolumn{4}{c}{\textbf{CTD}}  \\
\midrule
\textbf{image}                         & $\lambda$           & $\mu$                   & $\mathtt{it}$  & $\mathtt{it_{GS}}$ 
                                               & $\lambda_{min}$ & $\lambda_{max}$ & $\mathtt {it}$ & $\mathtt{it_{GS}}$ \\
\midrule
\texttt{cameraman}                 & $0.8$e$+3$  & $0.1$e$+3$  & 4 & 2.8 & $0.5$e$+3$  & $0.1$e$+4$ & 4 & 3   \\ 
\texttt{cameraman15}             & $0.3$e$+3$  & $0.1$e$+3$  & 5 & 4.2 & $0.3$e$+3$  & $0.1$e$+4$ & 5 & 4.2 \\ 
\texttt{cameraman25}             & $0.17$e$+3$ & $0.1$e$+3$  & 6 & 7   & $0.17$e$+3$ & $0.8$e$+3$ & 6 & 7   \\ 
\midrule
                                              & \multicolumn{4}{c}{\textbf{MM}}  & \multicolumn{4}{c}{\textbf{THR}} \\
\midrule
\textbf{image}                        & $\lambda_{min}$ & $\lambda_{max}$ & $\mathtt{it}$ & $\mathtt{it_{GS}}$
                                              & $\lambda_{min}$ & $\lambda_{max}$ & $\mathtt{it}$ & $\mathtt{it_{GS}}$ \\
\midrule
\texttt{cameraman}                & $0.5$e$+3$  & $0.1$e$+4$ & 4 & 3   & $0.5$e$+3$  & $0.1$e$+4$ & 4 & 2.8  \\
\texttt{cameraman15}            & $0.3$e$+3$  & $0.1$e$+4$ & 4 & 4.2 & $0.3$e$+3$  & $0.1$e$+4$ & 5 & 3.6  \\
\texttt{cameraman25}            & $0.17$e$+3$ & $0.5$e$+3$ & 6 & 7   & $0.17$e$+3$ & $0.8$e$+3$ & 7 & 5.7 \\
\bottomrule
\end{tabular}
\end{center}
\end{table}

\begin{figure}[p!]
\medskip
\begin{center}
\newcolumntype{C}{>{\centering\arraybackslash} m{.30\textwidth} }
\begin{tabular}{CCC}
\includegraphics[width=.24\textwidth]{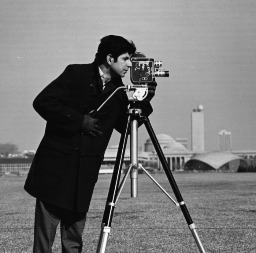} &
\includegraphics[width=.24\textwidth]{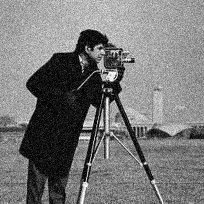} &
\includegraphics[width=.24\textwidth]{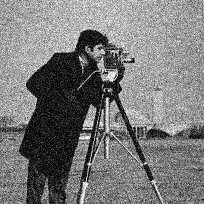} \\[1mm]
{\small \texttt{cameraman}}  &
{\small Gaussian noise ($\sigma = 15$)}  &
{\small Gaussian noise ($\sigma = 25$)} \\[1mm]
\includegraphics[width=.24\textwidth]{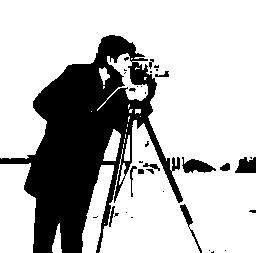} &
\includegraphics[width=.24\textwidth]{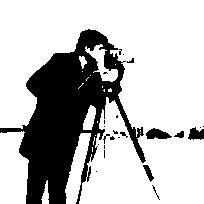} &
\includegraphics[width=.24\textwidth]{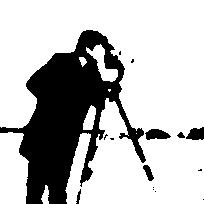} \\[1mm]
\multicolumn{3}{c}{\small CEN} \\ \hline \\[1mm]
\includegraphics[width=.24\textwidth]{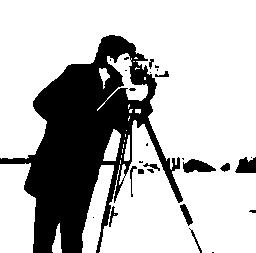} &
\includegraphics[width=.24\textwidth]{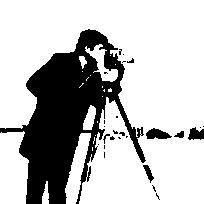} &
\includegraphics[width=.24\textwidth]{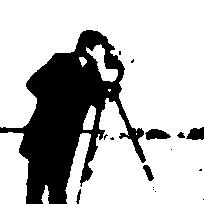} \\[1mm]
\multicolumn{3}{c}{\small CTD} \\ \hline \\[1mm]
\includegraphics[width=.24\textwidth]{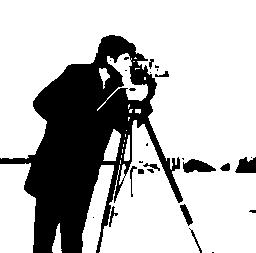} &
\includegraphics[width=.24\textwidth]{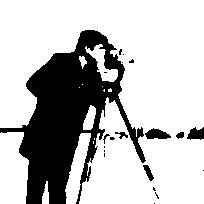} &
\includegraphics[width=.24\textwidth]{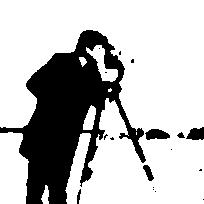} \\[1mm]
\multicolumn{3}{c}{\small MM} \\ \hline \\[1mm]
\includegraphics[width=.24\textwidth]{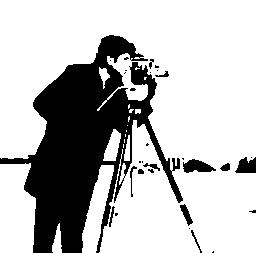} &
\includegraphics[width=.24\textwidth]{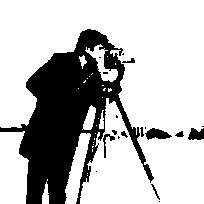} &
\includegraphics[width=.24\textwidth]{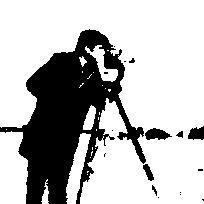} \\[1mm]
\multicolumn{3}{c}{\small THR} \\\hline
\end{tabular}
\caption{Segmentations of \texttt{cameraman} and its noisy versions by using the CEN, CTD, MM, and
THR models.\label{fig:noisy}}
\end{center}
\end{figure}

\clearpage

\authorcontributions{All the authors contributed equally to this work. They have read and agreed to the published
version of the manuscript.}

\funding{This work was partially supported by Istituto Nazionale di Alta Matematica - Gruppo Nazionale per il Calcolo Scientifico
(INdAM-GNCS), Italy. L.~Antonelli was also supported by the Italian Ministry of University and  Research under grant no. PON03PE\_00060\_5.}


\conflictsofinterest{The authors declare no conflict of interest. The funders had no role in the design of the study; in the collection, analyses,
or interpretation of data; in the writing of the manuscript, or in the decision to publish the results.}


\reftitle{References}
\externalbibliography{yes}
\bibliography{biblio_adaptiveparam}
\end{document}